\documentclass[12pt]{article}
\usepackage{amssymb}
\usepackage{amsthm}
\usepackage{amsmath}
\usepackage{color}
\usepackage[T1]{fontenc}

\newtheorem{definition}{Definition}
\newtheorem{lemma}[definition]{Lemma}

\newtheorem{theorem}[definition]{Theorem}

\newtheorem{example}[definition]{Example}
\newtheorem{corollary}[definition]{Corollary}
\title{Residuated operators in complemented posets\thanks{Preprint of an article accepted for publication in Asian-European Journal of Mathematics, \textcopyright\,World Scientific Publishing Company, \texttt{https://www.worldscientific.com/worldscinet/aejm}}}
\author{Ivan~Chajda and Helmut~L\"anger}
\date{}
\begin{document}
\footnotetext[1]{Support of the research of both authors by \"OAD, project CZ~04/2017, and IGA, project P\v rF~2018~012, and of the second author by the Austrian Science Fund (FWF), project I~1923-N25, is gratefully acknowledged.}
\maketitle
\begin{abstract}
Using the operators of taking upper and lower cones in a poset with a unary operation, we define operators $M(x,y)$ and $R(x,y)$ in the sense of multiplication and residuation, respectively, and we show that by using these operators, a general modification of residuation can be introduced. A relatively pseudocomplemented poset can be considered as a prototype of such an operator residuated poset. As main results we prove that every Boolean poset as well as every pseudo-orthomodular poset can be organized into a (left) operator residuated structure. Some results on pseudo-orthomodular posets are presented which show the analogy to orthomodular lattices and orthomodular posets.
\end{abstract}



Residuated structures play an important role in the algebraic axiomatization of some non-classical logics, in particular in so-called substructural logics and fuzzy logics. An algebraic theory of residuated structures has been developed in the last decades, see e.g.\ \cite{GJKO} for details. On the contrary, an algebraic semantic of the logic of quantum mechanics is provided by orthomodular lattices and orthomodular posets, see e.g.\ \cite B and the references there. The first attempts to get together residuated structures and so-called quantum structures were made in \cite{CH}. More sophisticated attempts to show that orthomodular lattices can be converted into left residuated structures were done by the authors in \cite{CL17a} and \cite{CL17b}. In fact, it was shown that every orthomodular lattice can be organized into left residuated l-groupoid. In order to extend these investigations to orthomodular posets, a more general approach is necessary. For example, so-called skew residuated lattices were studied in \cite{CK} and, in full generality, the concept of residuation was extended for relational structures in \cite{BC}. However, some researchers decided that the logic of quantum mechanics is based on a bit more general structures than orthomodular posets. Hence, a residuation of so-called weakly orthomodular and dually weakly orthomodular posets (\cite{CL18a}) and lattices (\cite{CL18b} has been investigated by the authors recently. Unfortunately, if an underlying lattice is replaced by a poset, residuation is possible only for certain elements which can be very restrictive, see \cite{CL14}. Hence, we introduce here a new concept, the so-called operator residuation, which seems to be more successful. Namely, we show that every relatively pseudocomplemented poset, every Boolean poset and every pseudo-orthomodular poset is operator residuated. The concept of a pseudo-orthomodular poset introduced here is very general, it includes orthomodular lattices as well as pseudo-Boolean posets and their horizontal sums. On the other hand, our paper does not contain a general theory of operator residuation but the authors believe that it would be an inspiration for other researchers to develop such a theory.

Let $\mathbf P=(P,\leq)$ be a poset. For $M\subseteq P$ we define
\begin{align*}
L(M) & :=\{x\in P\mid x\leq y\text{ for all }y\in M\}, \\
U(M) & :=\{x\in P\mid y\leq x\text{ for all }y\in M\}.
\end{align*}
We write $L(a)$ instead of $L(\{a\})$, $L(a,b)$ instead of $L(\{a,b\})$, $L(M,a)$ instead of $L(M\cup\{a\})$ and $L(M_1,M_2)$ instead of $L(M_1\cup M_2)$. Similarly for $U$. Let $a,b\in P$ and $A,B\subseteq P$. It is easy to show that the following are equivalent: $a\leq b$, $L(a)\subseteq L(b)$, $U(b)\subseteq U(a)$, $a\in L(a,b)$, $b\in U(a,b)$. Moreover, if $A\subseteq B$ then $L(B)\subseteq L(A)$ and $U(B)\subseteq U(A)$. If $a\vee b$ exists then $U(a,b)=U(a\vee b)$, if $a\wedge b$ exists then $L(a,b)=L(a\wedge b)$. 

The {\em poset} $\mathbf P$ is called {\em modular} if it satisfies one of the following equivalent conditions for all $x,y,z\in P$:
\begin{align*}
& x\leq z\text{ implies }L(U(x,y),z)=L(U(x,L(y,z))), \\
& x\leq z\text{ implies }U(x,L(y,z))=U(L(U(x,y),z)).
\end{align*}
Recall from \cite{CR} that the {\em poset} $\mathbf P$ is called {\em distributive} if it satisfies one of the following equivalent identities:
\begin{align*}
L(U(x,y),z) & \approx L(U(L(x,z),L(y,z))), \\
U(L(x,y),z) & \approx U(L(U(x,z),U(y,z))).
\end{align*}
Every distributive poset is modular, see \cite{CR}. If $(P,\leq)$ is a poset, $A\subseteq P$ and $'$ a unary operation on $P$ then we put $A':=\{x'\mid x\in A\}$. A {\em poset with complementation} is an ordered quintuple $\mathbf P=(P,\leq,{}',0,1)$ such that $(P,\leq,0,1)$ is a bounded poset and $'$ is a unary operation on $P$ satisfying the following conditions for all $x,y\in P$:
\begin{enumerate}
\item[(i)] $L(x,x')\approx\{0\}$ and $U(x,x')\approx\{1\}$,
\item[(ii)] $x\leq y$ implies $y'\leq x'$,
\item[(iii)] $(x')'\approx x$.
\end{enumerate}
If $(P,\leq)$ is a poset with a unary operation $'$ which is an antitone involution, i.e.\ which satisfies (ii) and (iii) from above, we can easily check that
\[
(U(x,y))'\approx L(x',y')\text{ and }(L(x,y))'\approx U(x',y'),
\]
which is a version of the De Morgan laws. A {\em Boolean poset} is a distributive poset with complementation.

\begin{example}
Fig.\ 1 shows two Boolean posets which are not lattices.

\vspace*{1.4mm}

\[
\setlength{\unitlength}{7mm}
\begin{picture}(6,8)
\put(3,0){\circle*{.3}}
\put(0,2){\circle*{.3}}
\put(2,2){\circle*{.3}}
\put(4,2){\circle*{.3}}
\put(6,2){\circle*{.3}}
\put(0,4){\circle*{.3}}
\put(2,4){\circle*{.3}}
\put(4,4){\circle*{.3}}
\put(6,4){\circle*{.3}}
\put(0,6){\circle*{.3}}
\put(2,6){\circle*{.3}}
\put(4,6){\circle*{.3}}
\put(6,6){\circle*{.3}}
\put(3,8){\circle*{.3}}
\put(3,0){\line(-3,2)3}
\put(3,0){\line(-1,2)1}
\put(3,0){\line(1,2)1}
\put(3,0){\line(3,2)3}
\put(3,8){\line(-3,-2)3}
\put(3,8){\line(-1,-2)1}
\put(3,8){\line(1,-2)1}
\put(3,8){\line(3,-2)3}
\put(0,2){\line(0,1)4}
\put(6,2){\line(0,1)4}
\put(0,4){\line(1,1)2}
\put(0,2){\line(1,1)4}
\put(2,2){\line(1,1)4}
\put(4,2){\line(1,1)2}
\put(2,2){\line(-1,1)2}
\put(4,2){\line(-1,1)4}
\put(6,2){\line(-1,1)4}
\put(6,4){\line(-1,1)2}
\put(2.85,-.75){$0$}
\put(-.6,1.9){$a$}
\put(1.2,1.9){$b$}
\put(4.45,1.9){$c$}
\put(6.4,1.9){$d$}
\put(-.6,3.9){$e$}
\put(1.2,3.9){$f$}
\put(4.45,3.9){$f'$}
\put(6.4,3.9){$e'$}
\put(-.7,5.9){$d'$}
\put(1.2,5.9){$c'$}
\put(4.45,5.9){$b'$}
\put(6.4,5.9){$a'$}
\put(2.85,8.4){$1$}
\put(2.6,-1.5){$(a)$}
\end{picture}
\quad\quad\quad\quad
\setlength{\unitlength}{7mm}
\begin{picture}(6,8)
\put(3,0){\circle*{.3}}
\put(0,2){\circle*{.3}}
\put(2,2){\circle*{.3}}
\put(4,2){\circle*{.3}}
\put(6,2){\circle*{.3}}
\put(0,4){\circle*{.3}}
\put(6,4){\circle*{.3}}
\put(0,6){\circle*{.3}}
\put(2,6){\circle*{.3}}
\put(4,6){\circle*{.3}}
\put(6,6){\circle*{.3}}
\put(3,8){\circle*{.3}}
\put(3,0){\line(-3,2)3}
\put(3,0){\line(-1,2)1}
\put(3,0){\line(1,2)1}
\put(3,0){\line(3,2)3}
\put(3,8){\line(-3,-2)3}
\put(3,8){\line(-1,-2)1}
\put(3,8){\line(1,-2)1}
\put(3,8){\line(3,-2)3}
\put(0,2){\line(0,1)4}
\put(6,2){\line(0,1)4}
\put(0,4){\line(1,1)2}
\put(0,2){\line(1,1)4}
\put(2,2){\line(1,1)4}
\put(4,2){\line(1,1)2}
\put(2,2){\line(-1,1)2}
\put(4,2){\line(-1,1)4}
\put(6,2){\line(-1,1)4}
\put(6,4){\line(-1,1)2}
\put(2.85,-.75){$0$}
\put(-.6,1.9){$a$}
\put(1.2,1.9){$b$}
\put(4.45,1.9){$c$}
\put(6.4,1.9){$d$}
\put(-.6,3.9){$e$}
\put(6.4,3.9){$e'$}
\put(-.7,5.9){$d'$}
\put(1.2,5.9){$c'$}
\put(4.45,5.9){$b'$}
\put(6.4,5.9){$a'$}
\put(2.85,8.4){$1$}
\put(2.6,-1.5){$(b)$}
\put(-2,-2){{\rm Fig.\ 1}}
\end{picture}
\]

\vspace*{9mm}

\end{example}

Recall from \cite{CL17a} or \cite{GJKO} that a {\em left residuated lattice} is an algebra $(L,\vee,\wedge,\odot,\rightarrow,1)$ of type $(2,2,2,2,0)$ satisfying the following conditions for all $x,y,z\in L$:
\begin{itemize}
\item $(L,\vee,\wedge)$ is a lattice,
\item $x\odot1\approx1\odot x\approx x$,
\item $x\odot y\leq z$ if and only if $x\leq y\rightarrow z$.
\end{itemize}
The last property is called {\em left adjointness}. If $\odot$ is commutative then this condition is called simply {\em adjointness}.

It was shown by the authors in \cite{CL17a} and \cite{CL17b} that every {\em orthomodular lattice} (see e.g.\ \cite B) can be converted into a left residuated lattice. A similar result was obtained by the authors for weakly orthomodular and dually weakly orthomodular lattices in \cite{CL18b}. Here we can define
\begin{align*}
      x\odot y & :=(x\vee y')\wedge y, \\
x\rightarrow y & :=(x\wedge y)\vee x'.
\end{align*}
In the case of Boolean algebras this reduces to
\begin{align*}
      x\odot y & :=x\wedge y, \\
x\rightarrow y & :=x'\vee y.
\end{align*}
Several attempts to convert Boolean or orthomodular posets into left residuated structures were made e.g.\ in \cite{CH} or \cite{CL14}, but in these cases left adjointness holds only for elements satisfying additional assumptions.

If we introduce the operations $\odot$ and $\rightarrow$ in Boolean posets in the same way as it was done in the case of Boolean algebras, i.e.\ $x\odot y:=x\wedge y$ and $x\rightarrow y:=x'\vee y$ if the corresponding meet and join exists, then the left adjointness property need not hold as the following example shows:

\begin{example}
Consider the Boolean poset depicted in Fig.\ 1 (a). Define $x\odot y:=x\wedge y$ and $x\rightarrow y:=x'\vee y$ whenever this meet and join exists. Assume that left adjointness is satisfied. Then
\begin{align*}
& b'\leq c'\rightarrow b'\text{ and hence }b'\odot c'\leq b', \\
& b'\leq c'\rightarrow c'\text{ and hence }b'\odot c'\leq c'.
\end{align*}
From $b'\odot c'\leq b'$ and $b'\odot c'\leq c'$ we conclude $b'\odot c'\in\{0,a,d\}$.
\begin{align*}
& \text{If }b'\odot c'=0\text{ then }b'\leq c'\rightarrow0=c,\text{ a contradiction}, \\
& \text{if }b'\odot c'=a\text{ then }b'\leq c'\rightarrow a=f,\text{ a contradiction}, \\
& \text{if }b'\odot c'=d\text{ then }b'\leq c'\rightarrow d=e',\text{ a contradiction}.
\end{align*}
Hence, we cannot go on in this way with partially defined operations $\odot$ and $\rightarrow$ and expect that the complemented poset can be converted into a left residuated structure. This is why we change our approach.
\end{example}

Since a poset with complementation has no binary operations one can hardly assume that it is possible to express the binary operations $\odot$ and $\rightarrow$ from left adjointness as term operations of $(P,\leq,{}',0,1)$. On the other hand, in every poset $(P,\leq)$ there are defined the s $L$ and $U$. Thus we can replace the operations $\odot$ and $\rightarrow$ by other s which express multiplication and residuation, respectively. We proceed as follows:

\begin{definition}
An {\em operator left residuated poset} is an ordered seventuple $\mathbf P=(P,\leq,{}',M,R,0,$ $1)$ where $(P,\leq,{}',0,1)$ is a bounded poset with a unary operation and $M$ and $R$ are mappings from $P^2$ to $2^P$ satisfying the following conditions for all $x,y,z\in P$:
\begin{eqnarray}
& & M(x,1)\approx M(1,x)\approx L(x),\label{equ1} \\
& & R(x,y)=P\text{ if and only if }x\leq y,\label{equ2} \\
& & M(x,y)\subseteq L(z)\text{ if and only if }L(x)\subseteq R(y,z),\label{equ3} \\
& & R(x,0)\approx L(x').
\end{eqnarray}
\end{definition}

Condition (\ref{equ3}) is clearly a generalization of {\em left-adjointness} and by (\ref{equ2}), $(P,\leq)$ can be reconstructed from $R$. If $M$ is commutative, i.e.\ $M(x,y)\approx M(y,x)$, then condition (\ref{equ3}) is simply called {\em adjointness} and we call $\mathbf P$ an {\em operator residuated poset}.

We can easily extend the operators $M$ and $R$ from $P^2$ to $(2^P)^2$, namely for $A,B\subseteq P$ we define
\[
M(A,B):=\bigcup_{(x,y)\in A\times B}M(x,y)\text{ and }R(A,B):=\bigcup_{(x,y)\in A\times B}R(x,y).
\]

We can now state the following lemma:

\begin{lemma}\label{lem1}
Every operator left residuated poset $(P,\leq,{}',M,R,0,1)$ with an antitone involution $'$ satisfies the identity $R(R(x,0),0)\approx P$.
\end{lemma}

\begin{proof}
We compute
\begin{align*}
R(R(x,0),0) & \approx R(L(x'),0)\approx\bigcup_{y\in L(x')}R(y,0)\approx\bigcup_{y\in L(x')}L(y')\approx\bigcup_{y'\in U(x)}L(y')\approx\bigcup_{z\in U(x)}L(z)\approx \\
            & \approx P.
\end{align*}
\end{proof}

Further, in residuated lattices the identity $(x\rightarrow y)\odot x\approx x\wedge y$ is usually called {\em divisibility}. Its modification for residuated posets could be as follows:
\begin{equation}\label{equ5}
M(R(x,y),x)\approx L(x,y).
\end{equation}
In the following, this identity will be referred to as {\em divisibility}.

A prototype of an operator residuated poset is a relatively pseudocomplemented one. Recall that a poset $(P,\leq)$ is called {\em relatively pseuducomplemented} if for each $a,b\in P$ there exists a greatest element $c$ of $P$ satisfying $L(a,c)\subseteq L(b)$, see e.g.\ \cite{Cu}. This element $c$ is called the {\em relative pseudocomplement} of $a$ with respect to $b$ and it is denoted by $a*b$. Every relative pseudocomplemented poset has a greatest element $1$ since $x*x=1$ for every $x\in P$.

\begin{theorem}\label{th2}
Let $\mathbf P=(P,\leq,*,0,1)$ be a bounded relatively pseudocomplemented poset. Define
\begin{align*}
M(x,y) & :=L(x,y), \\
R(x,y) & :=L(x*y)
\end{align*}
for all $x,y\in P$. Then $(P,\leq,^*,M,R,0,1)$ is an operator residuated poset satisfying divisibility. {\rm(}Here, for each $x\in P$, the element $x^*:=x*0$ denotes the {\em pseudocomplement} of $x$.{\rm)}
\end{theorem}

\begin{proof}
Let $a,b,c\in P$. \\
(1) We have $M(x,1)\approx L(x,1)\approx L(x)$ and $M(1,x)\approx L(1,x)\approx L(x)$. \\
(2) The following are equivalent: $R(a,b)=P$, $L(a*b)=L(1)$, $a*b=1$, $a\leq b$. \\
(3) The following are equivalent: $M(a,b)\subseteq L(c)$, $L(a,b)\subseteq L(c)$, $L(b,a)\subseteq L(c)$, $a\leq b*c$, $L(a)\subseteq L(b*c)$, $L(a)\subseteq R(b,c)$. \\
(4) We have $R(x,0)\approx L(x*0)\approx L(x^*)$. \\
(5) We have
\[
M(R(a,b),a)=\bigcup_{x\in R(a,b)}M(x,a)=\bigcup_{x\in L(a*b)}L(x,a).
\]
If $c\in\bigcup\limits_{x\in L(a*b)}L(x,a)$ then there exists some $d\in L(a*b)$ with $c\in L(d,a)$. Hence $L(a,d)\subseteq L(b)$ and $c\in L(d,a)$ whence $c\in L(a,b)$. If, conversely, $c\in L(a,b)$ then $L(a,c)\subseteq L(b)$ whence $c\in L(a*b,a)\subseteq\bigcup\limits_{x\in L(a*b)}L(x,a)$. This shows
\[
\bigcup_{x\in L(a*b)}L(x,a)=L(a,b)
\]
and hence
\[
M(R(a,b),a)=L(a,b).
\]
Since $M$ is commutative, we have operator adjointness.
\end{proof}

In the case of Boolean posets the operator $R(x,y)$ is constructed in a different way, see the next result.

\begin{theorem}\label{th1}
Let $(P,\leq,{}',0,1)$ be a Boolean poset and define
\begin{align*}
M(x,y) & :=L(x,y), \\
R(x,y) & :=L(U(x',y))
\end{align*}
for all $x,y\in P$. Then $(P,\leq,{}',M,R,0,1)$ is an operator residuated poset satisfying divisibility.
\end{theorem}

\begin{proof}
Let $a,b,c\in P$. \\
(1) is evident. \\
(2) If $a\leq b$ then $R(a,b)=L(U(a',b))\supseteq L(U(a',a))=L(1)=P$. If, conversely, $R(a,b)=P$ then $U(a',b)=\{1\}$ which implies
\begin{align*}
b\in U(0,b) & =U(L(a,a'),b)=U(L(U(a,b),U(a',b)))=U(L(U(a,b),1))= \\
            & =U(L(U(a,b)))=U(a,b),
\end{align*}
i.e.\ $a\leq b$. \\
(3) If $M(a,b)\subseteq L(c)$ then
\begin{align*}
L(a) & =L(1,a)=L(U(b,b'),a)=L(U(L(b,a),L(b',a)))\subseteq L(U(L(c),L(b',a)))\subseteq \\
     & \subseteq L(U(b',c))=R(b,c).
\end{align*}
If, conversely, $L(a)\subseteq R(b,c)$ then
\begin{align*}
M(a,b) & =L(a,b)\subseteq L(U(b',c),b)=L(U(L(b',b),L(c,b)))=L(U(0,L(c,b)))= \\
       & =L(U(L(c,b)))=L(c,b)\subseteq L(c).
\end{align*}
(4) We have $R(x,0)\approx L(U(x',0))\approx L(U(x'))\approx L(x')$. \\
(5 ) We have
\begin{align*}
M(R(x,y),x) & \approx M(L(U(x',y)),x)\approx\bigcup_{z\in L(U(x',y))}M(z,x)\approx\bigcup_{z\in L(U(x',y))}L(z,x)\approx \\
            & \approx L(U(x',y),x)\approx L(U(L(x',x),L(y,x)))\approx L(U(0,L(x,y)))\approx \\
            & \approx L(U(L(x,y)))\approx L(x,y).
\end{align*}
\end{proof}

It is worth noticing that in the case that $'$ is not a complementation, Theorem~\ref{th1} does not longer hold as can be seen from the following example:

\begin{example}
In the bounded distributive poset depicted in Fig.\ 2:
\[
\setlength{\unitlength}{7mm}
\begin{picture}(4,8)
\put(2,1){\circle*{.3}}
\put(1,3){\circle*{.3}}
\put(3,3){\circle*{.3}}
\put(1,5){\circle*{.3}}
\put(3,5){\circle*{.3}}
\put(2,7){\circle*{.3}}
\put(2,1){\line(-1,2)1}
\put(2,1){\line(1,2)1}
\put(1,3){\line(0,1)2}
\put(1,3){\line(1,1)2}
\put(3,3){\line(-1,1)2}
\put(3,3){\line(0,1)2}
\put(2,7){\line(-1,-2)1}
\put(2,7){\line(1,-2)1}
\put(1.85,.3){$0$}
\put(.4,2.85){$a$}
\put(3.3,2.85){$b$}
\put(.4,4.85){$c$}
\put(3.3,4.85){$d$}
\put(1.85,7.35){$1$}
\put(1.25,-.7){{\rm Fig.\ 2}}
\end{picture}
\]

\vspace*{5mm}

with
\[
\begin{array}{c|cccccc}
 x & 0 & a & b & c & d & 1 \\
\hline
x' & 1 & d & c & b & a & 0
\end{array}
\]
we have
\begin{align*}
M(c,d) & =L(c,d)=\{0,a,b\}\subseteq\{0,a,b,d\}=L(d), \\
  L(c) & =\{0,a,b,c\}\not\subseteq\{0,a,b,d\}=L(\{d,1\})=L(U(a,d))=L(U(d',d))=R(d,d)
\end{align*}
contradicting the operator adjointness property. Hence, it cannot be organized into an operator residuated poset in this way. On the other hand, this poset is relatively pseudocomplemented, the table for relative complementation is as follows:
\[
\begin{array}{c|cccccc}
* & 0 & a & b & c & d & 1 \\
\hline
0 & 1 & 1 & 1 & 1 & 1 & 1 \\
a & b & 1 & b & 1 & 1 & 1 \\
b & a & a & 1 & 1 & 1 & 1 \\
c & 0 & a & b & 1 & d & 1 \\
d & 0 & a & b & c & 1 & 1 \\
1 & 0 & a & b & c & d & 1
\end{array}
\]
Hence, by Theorem~\ref{th2}, it can be converted into operator residuated poset by using the residuation operator $R(x,y)=L(x*y)$.
\end{example}

The following concepts turn out to be useful for our investigations.

An {\em orthogonal poset} (cf.\ \cite{Cha}) is a poset $(P,\leq,{}',0,1)$ with complementation satisfying the following condition for all $x,y\in P$:
\[
\text{if }x\leq y'\text{ then }x\vee y\text{ exists}.
\]
Hence, if $x\leq y$ then $x\vee y'$ exists and, using De Morgan laws, also $y\wedge(x\vee y')=(y'\vee(x\vee y')')'$ exists. An {\em orthomodular poset} (cf. \cite B) is an orthogonal poset $(P,\leq,{}',0,1)$ satisfying one of the following equivalent identities ({\em orthomodular laws}):
\begin{align*}
((x\wedge y)\vee y')\wedge y & \approx x\wedge y, \\
  ((x\vee y)\wedge y')\vee y & \approx x\vee y.
\end{align*}
An example of an orthogonal poset which is not orthomodular is depicted in Fig.\ 3:
\[
\setlength{\unitlength}{7mm}
\begin{picture}(8,8)
\put(4,1){\circle*{.3}}
\put(1,3){\circle*{.3}}
\put(3,3){\circle*{.3}}
\put(5,3){\circle*{.3}}
\put(7,3){\circle*{.3}}
\put(1,5){\circle*{.3}}
\put(3,5){\circle*{.3}}
\put(5,5){\circle*{.3}}
\put(7,5){\circle*{.3}}
\put(4,7){\circle*{.3}}
\put(4,1){\line(-3,2)3}
\put(4,1){\line(-1,2)1}
\put(4,1){\line(1,2)1}
\put(4,1){\line(3,2)3}
\put(4,7){\line(-3,-2)3}
\put(4,7){\line(-1,-2)1}
\put(4,7){\line(1,-2)1}
\put(4,7){\line(3,-2)3}
\put(1,3){\line(0,1)2}
\put(1,3){\line(1,1)2}
\put(3,3){\line(-1,1)2}
\put(3,3){\line(0,1)2}
\put(5,3){\line(0,1)2}
\put(5,3){\line(1,1)2}
\put(7,3){\line(-1,1)2}
\put(7,3){\line(0,1)2}
\put(3.85,.3){$0$}
\put(.4,2.85){$a$}
\put(3.3,2.85){$b$}
\put(4.4,2.85){$c$}
\put(7.3,2.85){$d$}
\put(.4,4.85){$d'$}
\put(3.3,4.85){$c'$}
\put(4.4,4.85){$b'$}
\put(7.3,4.85){$a'$}
\put(3.85,7.35){$1$}
\put(3.25,-.7){{\rm Fig.\ 3}}
\end{picture}
\]

\vspace*{3mm}

In order to avoid problems with the existence of suprema and infima, we introduce the following concept.

\begin{definition}\label{def1}
A {\em pseudo-orthomodular poset} is a poset $\mathbf P=(P,\leq,{}',0,1)$ with complementation satisfying one of the following equivalent identities:
\begin{align*}
L(U(L(x,y),y'),y) & \approx L(x,y), \\
U(L(U(x,y),y'),y) & \approx U(x,y).
\end{align*}
\end{definition}

That e.g.\ the second identity follows from the first one can be seen by using the De Morgan laws:
\[
U(L(U(x,y),y'),y)\approx(L(U(L(x',y'),y),y'))'\approx(L(x',y'))'\approx U(x,y).
\]
Of course, if $\mathbf P$ is a lattice then these identities are equivalent to the orthomodular laws. Thus every orthomodular lattice is a pseudo-orthomodular poset. The identities mentioned in Definition~\ref{def1} can be weakened to inclusions as the following lemma shows:

\begin{lemma}\label{lem3}
In every poset $(P,\leq,{}',0,1)$ with complementation we have
\begin{align*}
L(U(L(x,y),y'),y) & \supseteq L(x,y), \\
U(L(U(x,y),y'),y) & \supseteq U(x,y)
\end{align*}
for all $x,y\in P$.
\end{lemma}

\begin{proof}
We have
\begin{align*}
L(U(L(x,y),y'),y) & \supseteq L(U(L(x,y)),y)\approx L(U(L(x,y)))\approx L(x,y), \\
U(L(U(x,y),y'),y) & \supseteq U(L(U(x,y)),y)\approx U(L(U(x,y)))\approx U(x,y)
\end{align*}
for all $x,y\in P$.
\end{proof}

We are going to show that the class of pseudo-orthomodular posets is not so small, namely a number of complemented posets turn out to belong to this class.

\begin{definition}
A {\em pseudo-Boolean poset} is a poset $\mathbf P=(P,\leq,{}',0,1)$ with complementation satisfying one of the following equivalent identities:
\begin{align*}
L(U(x,y),y') & \approx L(x,y'), \\
U(L(x,y),y') & \approx U(x,y').
\end{align*}
\end{definition}

That e.g.\ the second identity follows from the first one can be seen by using the De Morgan laws:
\[
U(L(x,y),y')\approx(L(U(x',y'),y))'\approx(L(x',y))'\approx U(x,y').
\]
Of course, every Boolean poset is pseudo-Boolean since in any Boolean poset we have
\[
L(U(x,y),y')\approx L(U(L(x,y'),L(y,y')))\approx L(U(L(x,y'),0))\approx L(U(L(x,y')))\approx L(x,y').
\]

\begin{lemma}\label{lem2}
Every pseudo-Boolean poset and hence also every Boolean poset is pseudo-orthomodular.
\end{lemma}

\begin{proof}
In any pseudo-Boolean poset we have
\[
L(U(L(x,y),y'),y)\approx L(U(x,y'),y)\approx L(x,y).
\]
\end{proof}

In the following we consider a construction of pseudo-orthomodular posets via so-called horizontal sums.

Let $\mathbf P_i=(P_i,\leq_i,{}'_i,0,1),i\in I$, be a family of bounded posets of cardinality greater than $2$ with a unary operation satisfying $0'_i=1$ and $1'_i=0$ and assume $P_i\cap P_j=\{0,1\}$ for all $i,j\in I$ with $i\neq j$. Put $P:=\bigcup\limits_{i\in I}P_i$. On $P$ we define a binary relation $\leq$ and a unary operation $'$ as follows:
\begin{align*}
& x\leq y\text{ if there exists some }i\in I\text{ such that }x\leq_iy\text{ in }P_i, \\
& x':=x'_i\text{ if }x\in P_i
\end{align*}
($x,y\in P$). Then $\mathbf P:=(P,\leq,{}',0,1)$ is well-defined and called the {\em horizontal sum } of the $\mathbf P_i,i\in I$. The $\mathbf P_i$ are called the {\em blocks} of $\mathbf P$.

\begin{lemma}\label{lem4}
The horizontal sum of pseudo-orthomodular posets is pseudo-orthomodular.
\end{lemma}

\begin{proof}
Let $\mathbf P_i=(P_i,\leq_i,{}'_i,0,1),i\in I$, be a family of pseudo-orthomodular posets of cardinality greater than $2$ satisfying $0'_i=1$ and $1'_i=0$ and assume $P_i\cap P_j=\{0,1\}$ for all $i,j\in I$ with $i\neq j$, let $\mathbf P=(P,\leq,{}',0,1)$ denote its horizontal sum and let $a,b\in P$. If $b=0$ then
\[
L(U(L(a,b),b'),b)=L(U(L(a,0),1),0)=\{0\}=L(a,0)=L(a,b).
\]
If $b=1$ then
\begin{align*}
L(U(L(a,b),b'),b) & =L(U(L(a,1),0),1)=L(U(L(a),0),1)=L(U(L(a)),1)= \\
                  & =L(U(L(a)))=L(a)=L(a,1)=L(a,b).
\end{align*}
If, finally, $b\neq0,1$ then there exists some $j\in I$ with $b\in P_j$. We then have
\[
L(a,b),U(L(a,b),b'),L(U(L(a,b),b'),b)\subseteq P_j.
\]
Since $\mathbf P_j$ is pseudo-orthomodular, we conclude $L(U(L(a,b),b'),b)=L(a,b)$. This shows that the identity $L(U(L(x,y),y'),y)\approx L(x,y)$ holds in $\mathbf P$.
\end{proof}

\begin{corollary}\label{cor1}
The horizontal sum of pseudo-Boolean posets is pseudo-orthomodular. Especially, the horizontal sum of Boolean posets is pseudo-orthomodular.
\end{corollary}

\begin{proof}
This follows from Lemmas~\ref{lem2} and \ref{lem4}.
\end{proof}

A pseudo-orthomodular poset need not be modular or orthomodular as the following example shows:

\begin{example}
Consider the horizontal sum $\mathbf P$ of the poset from {\rm Fig.\ 1 (b)} and an four-element Boolean algebra whose Hasse diagram is depicted in {\rm(Fig.\ 4)}:

\vspace*{3mm}

\[
\setlength{\unitlength}{7mm}
\begin{picture}(18,8)
\put(9,0){\circle*{.3}}
\put(6,2){\circle*{.3}}
\put(8,2){\circle*{.3}}
\put(10,2){\circle*{.3}}
\put(12,2){\circle*{.3}}
\put(6,4){\circle*{.3}}
\put(12,4){\circle*{.3}}
\put(6,6){\circle*{.3}}
\put(8,6){\circle*{.3}}
\put(10,6){\circle*{.3}}
\put(12,6){\circle*{.3}}
\put(9,8){\circle*{.3}}
\put(1,4){\circle*{.3}}
\put(17,4){\circle*{.3}}
\put(9,0){\line(-3,2)3}
\put(9,0){\line(-1,2)1}
\put(9,0){\line(1,2)1}
\put(9,0){\line(3,2)3}
\put(9,8){\line(-3,-2)3}
\put(9,8){\line(-1,-2)1}
\put(9,8){\line(1,-2)1}
\put(9,8){\line(3,-2)3}
\put(6,2){\line(0,1)4}
\put(12,2){\line(0,1)4}
\put(6,4){\line(1,1)2}
\put(6,2){\line(1,1)4}
\put(8,2){\line(1,1)4}
\put(10,2){\line(1,1)2}
\put(8,2){\line(-1,1)2}
\put(10,2){\line(-1,1)4}
\put(12,2){\line(-1,1)4}
\put(12,4){\line(-1,1)2}
\put(1,4){\line(2,-1)8}
\put(1,4){\line(2,1)8}
\put(17,4){\line(-2,-1)8}
\put(17,4){\line(-2,1)8}
\put(8.85,-.75){$0$}
\put(5.4,1.9){$a$}
\put(7.2,1.9){$b$}
\put(10.45,1.9){$c$}
\put(12.4,1.9){$d$}
\put(5.4,3.9){$e$}
\put(.4,3.9){$f$}
\put(12.4,3.9){$e'$}
\put(17.4,3.9){$f'$}
\put(5.3,5.9){$d'$}
\put(7.2,5.9){$c'$}
\put(10.45,5.9){$b'$}
\put(12.4,5.9){$a'$}
\put(8.85,8.4){$1$}
\put(8.2,-1.5){{\rm Fig.\ 4}}
\end{picture}
\]

\vspace*{8mm}

According to Corollary~\ref{cor1}, $\mathbf P$ is pseudo-orthomodular, but neither modular since $b\leq c'$, but
\begin{align*}
L(U(b,f),c') & =L(1,c')=L(c')=\{0,a,b,c\}\neq\{0,b\}=L(b)=L(U(b))= \\
& =L(U(\{0,b\},0))=L(U(L(b,c'),L(f,c'))).
\end{align*}
nor orthomodular since $b\leq c'$, but $b\vee c$ does not exist.
\end{example}

The following theorem describes a connection between pseudo-orthomodular posets and orthomodular posets.

\begin{theorem}
\
\begin{enumerate}
\item[{\rm(i)}] Every orthogonal pseudo-orthomodular poset is orthomodular.
\item[{\rm(ii)}] Every orthogonal modular poset with complementation is orthomodular.
\end{enumerate}
\end{theorem}

\begin{proof}
\
\begin{enumerate}
\item[(i)] Let $(P,\leq,{}',0,1)$ be an orthogonal pseudo-orthomodular poset and $a,b\in P$ and assume $a\leq b$. Then $L(a,b)=L(a)$ and $U(a,b')=U(a\vee b')$ and hence
\begin{align*}
L((a\vee b')\wedge b) & =L(a\vee b',b)=L(U(a\vee b'),b)=L(U(a,b'),b)=L(U(L(a),b'),b)= \\
                      & =L(U(L(a,b),b'),b)=L(a,b)=L(a)
\end{align*}
showing $(a\vee b')\wedge b=a$.
\item[(ii)] Let $(P,\leq,{}',0,1)$ be an orthogonal modular poset with complementation and $a,b\in P$ and assume $a\leq b$. Then, due to orthogonality, $a\vee b'$ exists and hence also $(a\vee b')\wedge b$ exists, and, using modularity, we compute
\begin{align*}
L((a\vee b')\wedge b) & =L(a\vee b',b)=L(U(a\vee b'),b)=L(U(a,b'),b)=L(U(a,L(b',b)))= \\
& =L(U(a,0))=L(U(a))=L(a)
\end{align*}
showing $(a\vee b')\wedge b=a$.
\end{enumerate}
\end{proof}

Now we treat a connection between modular posets with complementation and pseudo-orthomodular posets.

\begin{theorem}
A modular poset $(P,\leq,{}',0,1)$ with complementation is pseudo-or\-tho\-mo\-du\-lar if and only if
\[
L(\bigcap_{z\in L(x,y)}(U(z,y')\cup\{y\})\subseteq\bigcup_{z\in L(x,y)}L(U(z,y')\cup\{y\})
\]
for all $x,y\in P$.
\end{theorem}

\begin{proof}
Let $\mathbf P=(P,\leq,{}',0,1)$ be a modular poset with complementation. According to Lemma~\ref{lem3}, $\mathbf P$ is pseudo-orthomodular if and only if $L(U(L(x,y),y'),y)\subseteq L(x,y)$ for all $x,y\in P$. Now we have
\begin{align*}
L(U(L(x,y),y'),y) & \approx L(U(\bigcup_{z\in L(x,y)}\{z,y'\}),y)\approx L(\bigcap_{z\in L(x,y)}U(z,y'),y)\approx \\
                  & \approx L((\bigcap_{z\in L(x,y)}U(z,y'))\cup\{y\})\approx L(\bigcap_{z\in L(x,y)}(U(z,y')\cup\{y\}))
\end{align*}
and
\begin{align*}
L(x,y) & \approx\bigcup_{x\in L(x,y)}L(z)\approx\bigcup_{z\in L(x,y)}L(U(z))\approx\bigcup_{z\in L(x,y)}L(U(z,0))\approx \\
       & \approx\bigcup_{z\in L(x,y)}L(U(z,L(y',y)))\approx\bigcup_{z\in L(x,y)}L(U(z,y'),y)\approx\bigcup_{z\in L(x,y)}L(U(z,y')\cup\{y\}).
\end{align*}
\end{proof}

The next theorem shows that analogous to the corresponding result for orthomodular lattices, pseudo-orthomodular posets can be organized into operator left residuated posets.

\begin{theorem}
Let $(P,\leq,{}',0,1)$ be a pseudo-orthomodular poset and define
\begin{align*}
M(x,y) & :=L(U(x,y'),y), \\
R(x,y) & :=L(U(L(x,y),x'))
\end{align*}
for all $x,y\in P$. Then $\mathbf P=(P,\leq,{}',M,R,0,1)$ is an operator left residuated poset. Moreover, $\mathbf P$ satisfies divisibility if and only if it satisfies the identity
\[
\bigcup_{z\in R(x,y)}L(U(z,x'),x)\approx L(U(R(x,y),x'),x).
\]
\end{theorem}

\begin{proof}
Let $a,b,c\in P$. \\
(1) We have
\begin{align*}
M(x,1) & \approx L(U(x,0),1)\approx L(U(x))\approx L(x), \\
M(1,x) & \approx L(U(1,x'),x)=L(1,x)=L(x).
\end{align*}
(2) If $a\leq b$ then
\[
R(a,b)=L(U(L(a,b),a'))=L(U(L(a),a'))=L(U(a,a'))=L(1)=P.
\]
Conversely, if $R(a,b)=P$ then $U(L(a,b),a')=\{1\}$ and hence
\[
a\in L(a)=L(1,a)=L(U(L(b,a),a'),a)=L(b,a),
\]
i.e.\ $a\leq b$. \\
(3) If $M(a,b)\subseteq L(c)$, i.e.\ $L(U(a,b'),b)\subseteq L(c)$, then
\begin{align*}
L(a) & =L(U(a))\subseteq L(U(a,b'))=L(U(L(U(a,b'),b),b'))= \\
     & =L(U(L(b)\cap L(U(a,b'),b),b'))=L(U(L(b)\cap L(U(L(U(a,b'),b))),b'))= \\
		 & =L(U(L(b,U(L(U(a,b'),b))),b'))\subseteq L(U(L(b,U(L(c))),b'))= \\
     & =L(U(L(b,U(c)),b'))=L(U(L(b,c),b'))=R(b,c).
\end{align*}
Conversely, if $L(a)\subseteq R(b,c)$, i.e.\ $L(a)\subseteq L(U(L(b,c),b'))$, then
\begin{align*}
M(a,b) & =L(U(a,b'),b)=L(U(b',L(a)),b)\subseteq L(U(b',L(U(L(c,b),b'))),b)= \\
       & =L(U(b')\cap U(L(U(L(c,b),b'))),b)=L(U(b')\cap U(L(c,b),b'),b)= \\
       & =L(U(L(c,b),b'),b)=L(c,b)\subseteq L(c).
\end{align*}
(4) We have
\[
R(a,0)=L(U(L(a,0),a'))=L(U(0,a'))=L(U(a'))=L(a').
\]
(5) We have
\[
M(R(x,y),x)\approx\bigcup_{z\in R(x,y)}M(z,x)\approx\bigcup_{z\in R(x,y)}L(U(z,x'),x)
\]
and
\begin{align*}
L(x,y) & \approx L(U(L(x,y),x'),x)\approx L(U(L(U(L(x,y),x'))),x)\approx \\
& \approx L(U(L(U(L(x,y),x')),x'),x)\approx L(U(R(x,y),x'),x).
\end{align*}
\end{proof}

As shown in the paper, operator residuation can be useful if various posets are considered instead of lattices. Since several non-classical logics are based on underlying posets which need not be lattices, the question is how the operators $M(x,y)$ and $R(x,y)$ can be applied in the axiomatization of these logics. In particular, it can be of some interest how the logic of quantum mechanicas is related to pseudo-orthomodular posets. This particular question is connected with posets of certain self-adjoint operators in a Hilbert space. Up to now, we do not know any answers to these questions, but we have a strong believe that they will be a topic for the next study by the authors and possibly other interested researchers.

Authors' addresses:

Ivan Chajda \\
Palack\'y University Olomouc \\
Faculty of Science \\
Department of Algebra and Geometry \\
17.\ listopadu 12 \\
771 46 Olomouc \\
Czech Republic \\
ivan.chajda@upol.cz

Helmut L\"anger \\
TU Wien \\
Faculty of Mathematics and Geoinformation \\
Institute of Discrete Mathematics and Geometry \\
Wiedner Hauptstra\ss e 8-10 \\
1040 Vienna \\
Austria, and \\
Palack\'y University Olomouc \\
Faculty of Science \\
Department of Algebra and Geometry \\
17.\ listopadu 12 \\
771 46 Olomouc \\
Czech Republic \\
helmut.laenger@tuwien.ac.at
\end{document}